\documentclass[11pt,letterpaper,twoside]{article}
\usepackage[centertags]{amsmath}
\usepackage{graphicx,indentfirst,amsmath,amsfonts,amssymb,amsthm,newlfont,longtable}
\usepackage[toc,page]{appendix}
\usepackage{longtable,tikz,tikz-3dplot}
\usepackage{scalefnt}

\def\eps{\varepsilon}

\def\be{\beta}

\newcommand{\f}{\frac}

\newcommand{\ob}[2]{\overbrace{#1}^{\text{\hbox to 0cm{\hss #2 \hss}}}}
\newcommand{\ub}[2]{\underbrace{#1}_{\text{\hbox to 0cm{\hss #2 \hss}}}}
\def\lr#1#2#3{
    \left#1 #3 \right#2
}

\def\Ro{\mathcal{R}_0}
\def\Celsius{$^{\circ}$C}
\newcommand{\tSup}[1]{$^{#1}$}

\newcommand{\Q}{{\cal Q}_0}

\newcommand{\setaEsq}{\tikz \draw[-triangle 90] (0,0) -- +(-.1, .1);}
\newcommand{\setaDir}{\tikz \draw[-triangle 90] (0,0) -- +( .1,-.1);}
\newcommand{\retas}[2]{
    \draw[thick,gray] (#1-.3/4,#1-.5/4)  -- node{\setaEsq}
                      (#2,0);
    \draw[thick,gray] (#1-.3/4,#1-.5/4)  -- node{\setaDir}
                      (0,#2);}
  \def\figEum{\scalefont{0.7}
    \begin{tikzpicture}\phantom{$v^*$}
      \def\figWidth{4.5}
      \coordinate (1) at (0,0);
      \coordinate (2) at (\figWidth,0);
      \coordinate (3) at (\figWidth,\figWidth);
      \draw[thick,->] (1)  -- (\figWidth,0) node(\figWidth,0)[anchor=west]{$v^*$};
      \draw[thick,->] (1)  -- (0,\figWidth) node(0,\figWidth)[anchor=east]{$u^*$};
      \draw[thick] (1) -- (4.3-.3,4.3-.5) node[anchor=east]{$u^*=\f{\gamma v^*}{k\mu_1}\ $};
      \retas{.95}{1.7}
      \retas{1.6}{3}
      \retas{2.35}{4.4}
    \end{tikzpicture}
  }

\newcommand{\ds}{\displaystyle }
\newtheorem{definition}{Definition}
\newtheorem{theorem}{Theorem}
\newtheorem{lemma}{\bf Lemma}

\newtheorem{proposition}{Proposition}
\newtheorem{remark}{\bf Remark}

\newcommand{\p}{\partial}
\newcommand{\bb}{\begin{equation}}
\newcommand{\ee}{\end{equation}}
\newcommand{\ba}{\begin{array}}
\newcommand{\ea}{\end{array}}
\newcommand{\R}{\mathbb{R}}

\begin{document}
\pagenumbering{arabic}
\title{\huge \bf Mathematical modelling for the transmission of dengue: symmetry and traveling wave analysis}
\author{\rm \large Felipo Bacani$^1$, Stylianos Dimas$^2$, Igor Leite Freire$^2$, Norberto Anibal Maidana$^2$ and Mariano Torrisi$^3$ \\
\\
\it $^1$Instituto de Ciências Exatas e Aplicadas,\\
\it Universidade Federal de Ouro Preto - UFOP\\
    \it Rua Trinta e Seis, 
Loanda, $35931-008$ Jo\~ao Monlevade, MG - Brazil\\
    \rm E-mail: fbacani@decea.ufop.br\\
    \\
\it $^2$ Centro de Matem\'atica, Computa\c c\~ao e Cogni\c c\~ao\\ \it Universidade Federal do ABC - UFABC\\ \it 
Rua Santa Ad\'elia, $166$, Bairro Bangu,
$09.210-170$\\\it Santo Andr\'e, SP - Brasil\\
\rm E-mails: s.dimas@ufabc.edu.br, igor.freire@ufabc.edu.br/igor.leite.freire@gmail.com, \\norberto.maidana@ufabc.edu.br\\
\\
\it $^3$Dipartimento di Matematica e Informatica,
      \it  Universit\`a Degli Studi di Catania,\\
      \it Viale Andrea Doria, 6,
      95125    Catania, Italia;\\ 
      \rm E-mail\textbf{s}: torrisi@dmi.unict.it / m.torrisi12@gmail.com }
\date{\ }
\maketitle
\vspace{1cm}
\begin{abstract}
In this paper we propose some mathematical models for the transmission of dengue using a system of reaction-diffusion equations. The mosquitoes are divided into infected, uninfected and aquatic subpopulations, while the humans, which are divided into susceptible, infected and recovered, are considered homogeneously distributed in space and with a constant total population. We find Lie point symmetries of the models and we study theirs temporal dynamics, which provides us the regions of stability and instability, depending on the values of the basic offspring and the basic reproduction numbers. Also, we calculate the possible values of the wave speed for the mosquitoes invasion and dengue spread and compare them with those found in the literature.

\end{abstract}
\vskip 1cm
\begin{center}
{2010 AMS Mathematics Classification numbers: 34D20, 35B35, 76M60, 92Bxx 
\vspace{0.2cm}}
\end{center}
\pagenumbering{arabic}
\newpage
 
\section{Introduction}

The {\it Aedes aegypti} mosquito is a well known vector for the transmission of diseases to humans such as {\it dengue} and {\it Zika}, to name a few. Until 2015 these mosquitoes were mostly related to the transmission of dengue. However, evidence suggests that after 2014 FIFA World Cup tournament, Zika virus arrived at South America, finding in Brazil an ideal habitat to grow: a tropical climate, significantly higher population density and an efficient vector for transmission: {\it Aedes aegypti} \cite{ai,butler}. Zika usually causes mild symptoms in most people infected by it. In spite of everything,   new data gathered since the end of 2015  from women that got infected -- while they where on the last months of their pregnancy -- supported the suspicion that Zika is related to microcephaly, a medical condition where the baby's brain does not develop properly.

To the best of our knowledge no mathematical models of Zika have been proposed or validated so far \cite{howard}.  On the contrary,  things are quite different with dengue. For dengue, Aedes mosquito is the primary vector of transmission, and therefore, the study of its dynamics is very important as it permits the determination of the efficacy of different ways of controlling the mosquitoes populations. Furthermore, as a mosquito becomes a carrier of the virus only by biting an already infected human, the transmission can be fully understood only by taking also into account the human populations. On the other hand, for Zika,  this is only one of the possible ways of transmission since it can also be transmitted through other ways \cite{howard}.  Nevertheless, the study of dengue's transmission may be useful not only for its own sake, but it can also enlighten and provide insights and inspiration to the mathematical understanding of Zika too.

Our paper is concerned with the mathematical modelling for transmission of dengue. In section \ref{models} we propose Malthusian models taking into account a division of human population into three groups (SIR classification): susceptible, infectious and recovered, while the mosquitoes are divided into female winged non-infected and infected, and aquatic sub-populations.

To have a picture of some mathematical features of the biological constitutive parameters of the models considered we look for some point symmetries of the models in section \ref{symmetries}. Next, in the section \ref{spatial dynamics} we consider the temporal dynamics of the models. This enables us to determine the equilibrium points of the systems and determine whether these points are stable or not. In section \ref{speed}, using the invariance under space and time translations, we determine the wave speed for the mosquitoes' invasion and dispersion. To determine these values we made use of the data used in \cite{maidana2008}. Finally, discussions and conclusions are presented in section \ref{conclusion}.

\section{The models}\label{models}

We start by introducing the models for transmission of dengue relating humans and {\it Aedes} {\it aegypti} mosquitoes dynamics. 

The human population is divided into three sub-populations: susceptible, infected and recovered individuals at a time $t$ and a position $x$. The corresponding density functions are denoted by $\bar{h}(x,t)$, $\bar{I}(x,t)$ and $\bar{r}(x,t)$, respectively. By $\bar{N}(x,t)$ we designate the total human population, that is $\bar{N}(x,t)=\bar{h}(x,t)+\bar{I}(x,t)+\bar{r}(x,t)$.

The mosquitoes' population is also divided into three: winged non-infected $\bar{u}(x,t)$ and infected $\bar{w}(x,t)$ and aquatic $\bar{v}(x,t)$. The latter population includes the eggs, larvae and pupae stages of {\it Aedes} life cycle. The total winged mosquito population is denoted by $\bar{M}(x,t)$.

The biological parameters used in our models are presented on the Table \ref{tab1}.
\\
{\scalefont{0.8}
\begin{longtable}{|l|l|}
\caption{Biological parameters used for modelling.}\\
\hline\label{tab1}
Parameter          &     Biological meaning        \\\hline
 $\bar{\nu}$          & advection coefficient  \\\hline
 
 $\bar{r}_0$          & intrinsic oviposition rate  \\\hline
 
  $\bar{k}_1$          & carrying capacity regarding the winged mosquitoes form  \\\hline
 
 $\bar{k}_2$          & carrying capacity regarding the aquatic mosquitoes form   \\\hline
 
  $\bar{\gamma}$          & rate of maturation from the aquatic form of mosquitoes to winged form  \\\hline

 $\bar{\mu}_1$          & mortality rate of the winged mosquito sub-population  \\\hline
 
 $\bar{\mu}_2$          & mortality rate of the aquatic mosquito sub-population  \\\hline
 
  $\bar{\mu}_3$          & mortality rate of human population  \\\hline
 
 $\bar{\be}_1$          & transmission coefficient from humans to mosquitoes  \\\hline
 
 $\bar{\be}_2$          & transmission coefficient from mosquitoes to humans  \\\hline
 
  $\bar{\sigma}$          & recovery rate from disease  \\\hline
  $\bar{D}$   & diffusion coefficient, which may depend on the winged population\\\hline
\end{longtable}}

\subsection{Previous models}
Here we recall previous models that influenced this study. 

\subsubsection{{\it Aedes aegypti} population models}

As in the model proposed in \cite{maidana2005}, here we consider only two sub-populations: the winged form, comprised of mature female mosquitoes, and an aquatic sub-population, including eggs, larvae and pupae. The spatial density of the winged population is $\bar{M}(x,t)=\bar{u}(x,t)$ and the aquatic sub-population is $\bar{v}(x,t)$. The rate of maturation from the aquatic form to the winged one, denoted by $\bar{\gamma}$,  is satured by the carrying capacity $\bar{k}_1$, given by $\bar{\gamma}\bar{v}(x,t)\left(1-\bar{u}(x,t)/\bar{k}_1\right)$.

On the other hand, the rate of oviposition is proportional to the density of female mosquitoes, but it is also dependent on the availability of breeders, given by $\bar{r}_0\bar{u}(x,t)\left(1-\bar{v}(x,t)/\bar{k}_2\right)$.
Therefore, considering the parameters $(\bar{\nu}, \bar{\mu}_1,\bar{\mu}_2,\bar{k}_1,\bar{k}_2,\bar{r}_0)$ and the diffusion $\bar{D}=\bar{D}(u)$, one has the following mathematical model for the vital dynamics and dispersal process of mosquitoes:
\bb\label{2.2.1.1}
\left\{\ba{l}
\bar u_{\bar t}=(\ub{\bar D (u)\,\bar u_{\bar x})_{\bar x}}{diffusion}
         -\ob{\bar \nu\,\bar u _{\bar x}}{transport}
         +\ub{\bar \gamma\,\bar v\left(1-\f{\bar u}{k_{1}}\right)
             -\bar \mu_{1}\bar u}
             {birth/death},\\
             \\
             
\bar v_{\bar t}=\ob{\bar r_0\left(1-\f{\bar v}{k_2}\right)\bar M}
                  {oviposition}
         -\ob{(\bar \mu_{2}+\bar \gamma)\bar v}
             {\qquad lost eggs/eclosion}.
\ea\right.
\ee

Assuming that $\bar{D}(u)=\bar{D}$ is a constant, under the suitable non-dimensional transformation
\bb\label{2.2.1.2}
\ba{l}\ds{
   u=\f{\bar u}{k_1},\,\,
    v=\f{\bar v}{k_2},\,\,
    t=\bar t\cdot \bar r_0,\,\,
    x=\f{\bar x}{\sqrt{\bar D/\bar r_0}}}\,\,
   \ds{ \mu_1 = \f{\bar \mu_1}{\bar r_0},\,\,
    \mu_2 = \f{\bar \mu_2}{\bar r_0},\,\,
    \gamma= \f{\bar \gamma}{\bar r_0},\,\,
    \nu   = \f{\bar \nu}{\sqrt{\bar r_0\bar D}},\,\,
    k     = \f{k_1}{k_2},  }
\ea
\ee
one can transform system (\ref{2.2.1.1}) into
\bb\label{2.2.1.3}
\left\{
\ba{lcl}
u_{t}&=&\ds{u_{xx}-\nu u_{x}+\f{\gamma}{k}v(1-u)-\mu_{1}u},\\
\\
v_{t}&=&\ds{k(1-v)u-(\mu_{2}+\gamma)v}.
\ea\right.
\ee

If one assumes nonlinear effects in the diffusion (in section \ref{new} we shall revisit this point) of the type $\bar{D}(u)\sim\bar{u}^p$, it is induced a nonlinear effect on the transport. Then it may be of interest the addition of nonlinear effects in both diffusion and transport. These nonlinear effects can be considered in (\ref{2.2.1.3}) by making the changes $u_{xx}\mapsto (u^pu_{x})_{x}$ and $\nu u_x\mapsto 2\nu u^q u_x$, respectively, where $p$ and $q$ are arbitrary parameters. Additionally, if one removes the species' self-regulation term $uv$ in (\ref{2.2.1.3}) and add $(\gamma/k)u$ (for further details, see \cite{frema2013}), one removes the mosquitoes' saturation. Thus, the following model is obtained:
\bb\label{2.2.1.4}
\left\{
\ba{l}
\ds{u_{t}=(u^{p}u_{x})_x -2\nu u^{q}u_{x} +\f{\gamma}{k}v+(\f{\gamma}{k}-\mu_{1})u},\\
\\
\ds{v_{t}=ku+(k -\mu_{2}-\gamma)v}.
\ea\right.
\ee

\subsubsection{Model for transmission of dengue to humans via {\it Aedes aegypti}}

In \cite{maidana2008}, assuming a constant diffusion  for transmission of dengue and taking human and mosquito populations into account, the following model has been proposed:
\bb\label{2.2.2.1}
\left\{\ba{l}
\bar u_{\bar t}=\ub{\bar D\,\bar u_{\bar x\bar x}}{diffusion}
         -\ob{\bar \nu\,\bar u _{\bar x}}{transport}
         +\ub{\bar \gamma\,\bar v\left(1-\f{\bar M}{k_{1}}\right)
             -\bar \mu_{1}\bar u}
             {birth/death}
         -\ob{\bar \beta_{1}\bar u \bar I}
             {infection human$ \rightarrow $mosquito},\\
             \\
             
\bar w_{\bar t}=\bar D\,\bar w_{\bar{xx}}
         -\bar \nu \,\bar w_{\bar x}
         -\bar \mu_{1}\,\bar w+\bar \beta_{1}\bar u \bar I,\\
         \\
\bar v_{\bar t}=\ob{\bar r_0\left(1-\f{\bar v}{k_2}\right)\bar M}
                  {oviposition}
         -\ob{(\bar \mu_{2}+\bar \gamma)\bar v}
             {\qquad lost eggs/eclosion},\\
             \\
\bar h_{\bar t}=\ub{\bar \mu_3\bar N
                  -\bar \mu_3\bar h}{birth/death}
              -\ob{\bar \beta_{2}\bar h\bar w}{infection mosquito$ \rightarrow $human},\\
              \\
\bar I_{\bar t}=\bar \beta_{2}\bar h\bar w
               -\ub{\bar \sigma \bar I}{natural recovery}
               -\ob{\bar \mu_3\bar I}{mortality},\\
         \\
\bar r_{\bar t}=\bar \sigma \bar I
              -\bar \mu_3\bar r.
\ea\right.
\ee
The last three equations of the previous system yields (see \cite{maidana2008})
$$\frac{\partial \bar N}{\partial \bar t}\equiv\frac{\partial\bar h}{\partial\bar t} +\frac{\partial \bar I}{\partial \bar t}+\frac{\partial \bar r}{\partial\bar t}=0$$
which implies on the constancy of the human population, although each sub-population may vary, {\it e.g}, due to natality, mortality or other events. Therefore $\bar N = \bar N_0$ is a constant, where $\bar N_0$ is the population at $t=0$.
\par
By using (\ref{2.2.1.2}) together with
\bb\label{2.2.2.2}
\ba{l}
{\ds   w=\f{ \bar{w}}{k_1},\,\,
    h=\f{\bar{h}}{\bar N},\,\,
    I=\f{\bar{I}}{\bar N},\,\,
   r=\f{ \bar{r}}{\bar N}},\,\,
{\ds    \beta_1 = \f{\bar \beta_1\,\bar N}{\bar r_0},\,\,
    \beta_2 = \f{\bar \beta_2\,k_1}{\bar r_0},\,\,
    \mu_3   = \f{\bar \mu_3}{\bar r_0},\,\,
    \sigma  = \f{\bar \sigma}{\bar r_0}}
\ea
\ee
 the system is put (\ref{2.2.2.1}) in the non-dimensional form
\bb\label{2.2.2.3}
\left\{\ba{l}
u_{t}=u_{xx}-\nu u_{x}+\f{\gamma}{k}v\left(1-M\right) -\mu_{1}u-\beta_{1}u I,\\
\\
w_{t}=w_{xx}-\nu w_{x}-\mu_{1}w+\beta_{1}u I,\\
\\
v_{t}=k(1-v)M -(\mu_{2}+\gamma)v,\\
\\
h_{t}=(1-h)\mu_3-\beta_{2}hw,\\
\\
I_{t}=\beta_{2}hw-\sigma I-\mu_3 I,\\
\\
r_{t}=\sigma I-\mu_3 r,
\ea\right.
\ee
and, in particular, $ h+I+r=1$.
\subsection{New models}\label{new}

Differently from \cite{maidana2005,maidana2008}, in what follows we assume that the diffusion of the winged population is dependent on the density $\bar{M}$. Actually, we make the assumption $\bar{D}\propto M^p$, where $p$ is a parameter. Nonlinearities in the diffusion may be of particular interest in phenomena in which population density is relevant. A typical dependence on the population density is given by
\bb\label{2.3.1}
D(\bar{M})=\bar{D}_0\left(\f{\bar{M}}{M_0}\right)^p,
\ee
where $\bar{D}_0$ is a constant and $M_0$ is usually understood as the carrying support capacity of the population or the initial population. Usually $p\geq0$, since with this choice we have $d\bar{D}/dM\geq0$, which implies that the diffusion increases with the population. However, in this paper we do not impose such a restriction, leaving the parameter $p$ arbitrary, which leads us to a richer mathematical problem. We would like to note here that the diffusion of insects is an important phenomena yet to be fully understood, with only a few works considering nonlinear effects on the diffusion.
 
 By assuming a nonlinear diffusion, nonlinear effects of the type $M^{p-1}M_uu_x$ and $M^{p-1}M_w w_x$ may contribute to the advection. Therefore, we also assume that the advection terms depend on power nonlinearities of the populations. This is a mathematical assumption. On the other hand, it is worth noticing that we shall carry out a symmetry classification of the models proposed in the next section. Symmetry classifications of systems with several parameters, as is our case, may be influenced by certain constraints involving the parameters of the equations under consideration. Quite frequently, the special cases appearing during the classification of symmetry groups have significance in the physical process involved, see, for instance, \cite{olver}, exercise 2.18. For this reason we add these nonlinearities {\it a priori} in the models, leaving a possible interpretation {\it a posteriori}, after the symmetries are found.

\begin{enumerate}
\item {\bf Model 1}: Making in (\ref{2.2.2.1}) the corresponding modifications applied to (\ref{2.2.1.1}) in order to obtain (\ref{2.2.1.4}) and invoking (\ref{2.2.1.2}) and (\ref{2.2.2.2}), one has
\bb\label{2.3.2}
\left\{\ba{l}
\ds{u_{t}=\left(M^pu_{x}\right)_x -2\nu u^{q_{1}}u_{x} +\f{\gamma}{k}v +\left(\f{\gamma}{k}-\mu_{1}\right)u-\beta_{1} u I,}\\
\\
\ds{w_{t}=\left(M^p w_{x}\right)_x -2\nu w^{q_{2}}w_{x}+\left(\f{\gamma}{k}-\mu_{1}\right)w+\beta_{1}u I,}\\
\\
v_{t}= k M +(k -\mu_{2}-\gamma)v,\\
\\
h_{t}=(1-h)\mu_3-\beta_{2}hw,\\
\\
I_{t}=\beta_{2}hw-\sigma I-\mu_3I,\\
\\
r_{t}=\sigma I-\mu_3r.
\ea\right.
\ee
System (\ref{2.3.2}) can be considered a generalisation of the models studied in \cite{frema2013,frema2014}, incorporating humans and splitting the winged population into infected and non-infected.

\item {\bf Model 2}: An additional Malthusian model can be obtained by removing the saturation of mosquitoes and eggs (\ref{2.2.2.3}). This hypothesis yields the following system: 
\bb\label{2.3.3}
\left\{\ba{l}
{\ds u_{t}=\left(M^pu_{x}\right)_x-2\nu u^{q_{1}}u_{x}+\f{\gamma}{k}v-\mu_{1}u-\beta_{1} u I},\\
\\
\ds{w_{t}=\left(M^p w_{x}\right)_x-2\nu w^{q_{2}}w_{x}-\mu_{1}w+\beta_{1}u I,}\\
\\
v_{t}= k M -(\mu_{2}+\gamma)v,\\
\\
h_{t}=(1-h)\mu_3-\beta_{2}hw,\\  
\\
I_{t}=\beta_{2}hw-\sigma I-\mu_3I,\\
\\
r_{t}=\sigma I-\mu_3r.
\ea\right.
\ee
\end{enumerate}

Quite interesting, systems (\ref{2.3.2}) and (\ref{2.3.3}) are members of the family of systems
\bb\label{2.3.4}
\left\{\ba{l}
\Delta_1:=\ds{u_{t}-\left(M^pu_{x}\right)_x+2\nu u^{q_{1}}u_{x}-\f{\gamma}{k}v-(\epsilon\f{\gamma}{k}-\mu_{1})u+\beta_{1} u I}=0,\\
\\
\Delta_2:=\ds{w_{t}-\left(M^p w_{x}\right)_x+2\nu w^{q_{2}}w_{x}-(\epsilon\f{\gamma}{k}+\mu_{1})w-\beta_{1}u I}=0,\\
\\
\Delta_3:=v_{t} - k M +(\mu_{2}+\gamma-\epsilon k)v=0,\\
\\
\Delta_4:=h_{t}-(1-h)\mu_3+\beta_{2}hw=0,\\  
\\
\Delta_5:=I_{t}-\beta_{2}hw+\sigma I+\mu_3I=0,\\
\\
\Delta_6:=r_{t}-\sigma I+\mu_3r=0.\\
\ea\right.
\ee



\section{Lie symmetries of the system (\ref{2.3.4})}\label{symmetries}

In this section we investigate point symmetries of the system (\ref{2.3.4}). As already mentioned, such analysis in mathematical models with several parameters usually reveals those who are really important. In our case, it may enlighten our knowledge on the biological parameters and this is of interest for mathematical studies of some system of type (\ref{2.3.4}). Moreover they are useful to understand how the mathematical structure of the system could be modified in order to improve the fitting of the model with the real phenomenon. 

Although it may not have biological meaning for all values of $\epsilon$, system (\ref{2.3.4}) contains both systems (\ref{2.3.2}) and (\ref{2.3.3}) as members, and from the point of view of Lie symmetries the effort for determining the invariance group of either (\ref{2.3.2}) or (\ref{2.3.3}) is the same of (\ref{2.3.4}). Therefore, here,  we focus our attention in  (\ref{2.3.4}).

We shall proceed in the following way: first we give a short overview on Lie point symmetries and then, we find symmetries of (\ref{2.3.4}).

\subsection{Lie point symmetries}

\begin{definition}
	A continuous one-parameter (local) Lie group of transformations is a family $G$ 
	\begin{equation}\label{contT}
		T_{\hat\epsilon}:=\left\{
		\begin{array}{l}
		x^*=x^{*}(x,t,u,w,v,h,I,r,\hat\epsilon ),
		\,\, 
		t^*=t^{*}(x,t,u,w,v,h,I,r,\hat\epsilon),
		\\ 
		u^*=u^{*}(x,t,u,w,v,h,I,r,\hat\epsilon),
		\,\,
		w^*=w^{*}(x,t,u,w,v,h,I,r,\hat\epsilon),
		\\ 
		v^*=v{*}(x,t,u,w,v,h,I,r,\hat\epsilon),
		\,\,
		h^*=h^{*}(x,t,u,w,v,h,I,r,\hat\epsilon),
		\\ 
		I^*=I^{*}(x,t,u,w,v,h,I,r,\hat\epsilon),
		\,\,
		r^*=r^{*}(x,t,u,w,v,h,I,r,\hat\epsilon),	
		\end{array}\right.
		\end{equation}
	which is locally a $C^{\infty}$-diffeomorphism in a subset $S\subseteq \mathbb{R}^{2+6}$  with coordinates  $(x,t,u,w,v,h,I,r)$, depending analytically on
	the parameter $\hat\epsilon$ in a neighbourhood  $D \subseteq \mathbb{R}$  of $\hat\epsilon=0$ and
	reduces to the identity transformation when $\hat\epsilon=0$. A Lie point symmetry for the system  $(\ref{2.3.4})$ is a transformation $(\ref{contT})$ leaving $(\ref{2.3.4})$ invariant.
	
\end{definition}
By expanding with respect to $\hat\epsilon$ around $0$ we get the linear form  of (\ref{contT})
\begin{equation}\label{transfdef}
	\begin{array}{l}
		x^*=x+\hat\epsilon\xi(x,t,u,w,v,h,I,r) +O(\hat\epsilon^2), \,\,
		t^*=t+\hat\epsilon\tau(x,t,u,w,v,h,I,r) +O(\hat\epsilon ^2),\\
		\\
		u^*=u+\hat\epsilon\eta^{1}(x,t,u,w,v,h,I,r)+O(\hat\epsilon ^2),\,\,
		w^*=w+\hat\epsilon\eta^{2}(x,t,u,w,v,h,I,r)+O(\hat\epsilon ^2),\\ 
		\\
				v^*=v+\hat\epsilon\eta^{3}(x,t,u,w,v,h,I,r)+O(\hat\epsilon ^2),\,\,
				h^*=h+\hat\epsilon\eta^{4}(x,t,u,w,v,h,I,r)+O(\hat\epsilon ^2),\\ 
				\\
			I^*=I+\hat\epsilon\eta^{5}(x,t,u,w,v,h,I,r)+O(\hat\epsilon ^2),\,\,
			r^*=r+\hat\epsilon\eta^{6}(x,t,u,w,v,h,I,r)+O(\hat\epsilon ^2),			
	\end{array}
\end{equation}
where
$$ \xi(x,t,u,w,v,h,I,r):=\frac{\partial x^{*}}{\partial \hat\epsilon}\left|_{\hat\epsilon=0}\right.,\quad\quad\quad\tau(x,t,u,w,v,h,I,r):=\frac{\partial t^{*}}{\partial \hat\epsilon}\left|_{\hat\epsilon=0}\right.,$$
$$\eta^{1}(x,t,u,w,v,h,I,r):=\frac{\partial u^{*}}{\partial \hat\epsilon}\left|_{\hat\epsilon=0}\right.,\quad\quad\quad \eta^{2}(x,t,u,w,v,h,I,r):=\frac{\partial w^{*}}{\partial \hat\epsilon}\left|_{\hat\epsilon=0}\right., $$
$$\eta^{3}(x,t,u,w,v,h,I,r):=\frac{\partial v^{*}}{\partial \hat\epsilon}\left|_{\hat\epsilon=0}\right.,\quad\quad\quad \eta^{4}(x,t,u,w,v,h,I,r):=\frac{\partial h^{*}}{\partial \hat\epsilon}\left|_{\hat\epsilon=0}\right., $$
$$\eta^{5}(x,t,u,w,v,h,I,r):=\frac{\partial I^{*}}{\partial \hat\epsilon}\left|_{\hat\epsilon=0}\right.,\quad\quad\quad \eta^{6}(x,t,u,w,v,h,I,r):=\frac{\partial r^{*}}{\partial \hat\epsilon}\left|_{\hat\epsilon=0}\right., $$
allows us to introduce the vector field
\bb\label{1.1}
\ba{lcl}
X&=&\ds{\xi(x,t,u,w,v,h,I,r)\f{\p}{\p x}+\tau(x,t,u,w,v,h,I,r)\f{\p}{\p t}+\eta^1(x,t,u,w,v,h,I,r)\f{\p}{\p u}}\\
\\
&&\ds{+\eta^2(x,t,u,w,v,h,I,r)\f{\p}{\p w}+\eta^3(x,t,u,w,v,h,I,r)\f{\p}{\p v}+\eta^4(x,t,u,w,v,h,I,r)\f{\p}{\p h}}\\
\\
&&\ds{+\eta^5(x,t,u,w,v,h,I,r)\f{\p}{\p I}+\eta^6(x,t,u,w,v,h,I,r)\f{\p}{\p r}}.
\ea
\ee

This operator is usually called infinitesimal generator of the transformation or infinitesimal generator of the Lie point symmetry. Then, given a transformation (\ref{transfdef}), it is possible to obtain the corresponding generator (\ref{1.1}). Vice-versa, given a generator of the type (\ref{1.1}), it is possible to obtain its transformation using the exponential map, that is, the transformation is given by $(x,t,u,w,v,h,I,r)\mapsto (e^{\hat\epsilon X}x,e^{\hat\epsilon X}t,e^{\hat\epsilon X}u,e^{\hat\epsilon X}w,e^{\hat\epsilon X}v,e^{\hat\epsilon X}h,e^{\hat\epsilon X}I,e^{\hat\epsilon X}r)$.
\par
In order to obtain the symmetries of system (\ref{2.3.4}) one should extend the operator (\ref{1.1}) up to second order and then apply the invariance condition (see the well known references  \cite{bk,i2,i1,olver} for further details) that reads:
\begin{equation} \label{INV}
\left\{
\begin{array}{l}
 X^{(2)}\Delta_1=0\Big |_\mathbf{\Delta=0},\quad  X^{(2)}\Delta_2=0\Big |_\mathbf{\Delta=0},\quad  X^{(1)}\Delta_3=0\Big |_\mathbf{\Delta=0},\quad
 \\\\  
 X^{(1)}\Delta_4=0\Big |_\mathbf{\Delta=0},\quad  X^{(1)}\Delta_5=0\Big |_\mathbf{\Delta=0},\quad  X^{(1)}\Delta_6=0\Big |_\mathbf{\Delta=0}, 
 \end{array}\right.
\end{equation}
where $\mathbf{\Delta}:=(\Delta_1,\,\Delta_2,\,\Delta_3,\, \Delta_4,\, \Delta_5,\, \Delta_6)^T$, being $X^{(1)}$ and $X^{(2)}$ the first and the second extensions of generator $X$, respectively.

By using the symbolic package SYM for Mathematica\textregistered\ deloped by SD, see \cite{dimas1,dimas2},  we obtain from (\ref{INV}) \textit{the determining system}  (see again \cite{bk,i2,i1,olver} for further details). The solutions of such a system provide the components $\xi,\tau, \eta^1, \eta^2,\eta^3,\eta^4,\eta^5,\eta^6$ of the generator (\ref{1.1}). The Principal Lie Algebra $L_{\cal P}$, {\it i.e}, those symmetries leaving the system invariant for all parameters, is spanned by $X=\partial_{x},\,T=\partial_{t}$, which correspond to the generators of translations in space and time.

{\begin{table}[!h]\centering
 \scalefont{0.7}
\caption{ \scalefont{0.9}
Some symmetries of the system (\ref{2.3.4}). A more complete list of symmetries of (\ref{2.3.4}) is presented in the chapter 4 of the references \cite{felipo-tese}. The biological parameters $\be_1,\,\be_2,\,\nu,\,\mu_3$ and $\sigma$ are non-negatives, while the nonlinearities $p,\,q_1$ and $q_2$ may assume any real value.}
\begin{tabular}{|l|l|l|l|l|l|l|l|l|}
\hline\label{tab2}
Case&  $\beta_1$  &     $p$    &   $\beta_2$    &    $\nu$    &   $q_1,\,q_2$  &   $\mu_3$    &    $\sigma$    & Extensions with respect to $L_{\cal P}$\\\hline
  1\phantom'&$\forall$&$0$&$0$&$\forall$&$0$&$\forall$&$\forall$&
  $X_{2}=u\partial_{u} + w\partial_{w} + v\partial_{v}$\\\hline
 
  2\phantom'&$\forall$&$\forall$&$0$&$\forall$&$p/2$&$\forall$&$\forall$&
  $X_{3}=p\,x\,\partial_{x}  +  2\lr(){u\partial_{u} + w\partial_{w} + v\partial_{v}}$\\\hline
   3\phantom'&$\forall$&$\forall$&$0$&$0$&$\forall$&$\forall$&$\forall$&
  $X_{3}=p\,x\,\partial_{x}  +  2\lr(){u\partial_{u} + w\partial_{w} + v\partial_{v}}$\\\hline
  
  4 & $0$ & $0$ & $\neq0$ & $0$ & $\forall$ & $\forall$ & $\forall$ & $X_4=(u+w)\p_u+v\p_v$ \\\hline
 
5 & $0$ & $0$ & $\neq0$ & $0$ & $\forall$ & $\forall$ & $\forall$ & $X_5=w\p_{u}-w\p_{w}$ \\\hline

 6 & $0$ & $\forall$ & $0$ & $\forall$ & $\forall$ & $\forall$ & $\forall$ & $X_\infty=f\p_I+g\p_r\,(f,g)\,\text{is a solution of}\,\Delta_5=0,\,\Delta_6=0$ \\\hline 
 

  7 & $0$ & $0$ & $0$ & $0$ & $\forall$ & $\forall$ & $\forall$ & $X_\infty=f_1\p_{u}+f_2\p_{w}+f_3\p_v+f_4\p_h+f_5\p_I+f_6\p_r,\,(f_1,\cdots,f_6)\,\text{is a solution of}\,\mathbf{\Delta}=0$ \\\hline
    \end{tabular}
\end{table}

As usually occurs in empirical mathematical models, the search for solutions of the determining system, or more precisely, the search for the symmetries, is quite complex as they depend on the biological parameters $\beta_1$, $\beta_2$, $\nu$, $\mu_3$, $\sigma$, jointly with values ensuring nonlinearities $p,\,q_1$ and $q_2$. A more complete list of symmetries of (\ref{2.3.4}) would take a considerable amount of space and is decomposed in several cases and sub-cases, much of them without biological relevance nor meaning. Hence we opt to show only some of them. On Table \ref{tab2} we present some extensions of the Principal Lie Algebra. In the reference \cite{felipo-tese} the reader can find several pages reporting the classification of symmetries of (\ref{2.3.4}).

\subsection{A case of biological relevance}

A case of biological relevance occurs when $\be_1\be_2\sigma\neq0$, which implies that the interaction between mosquitoes and humans is present, as well as there are infected humans recovering from the disease. Moreover, we also consider  $\mu_3=0$, a condition expressing the fact that no human dies, which may occur if a short period of time is considered. 
Therefore, considering a linear combination of the generators of the Principal Lie Algebra given by $cX+T=c\p_x+\p_t$, from the invariant form method (see \cite{bk}, page 197) we obtain the following invariants: $z=x-ct$ and
\bb\label{3.4.3}
u=\Phi_1(z),\, w=\Phi_2(z),\, v=\Phi_3(z),\,
h=\Phi_4(z),\, I=\Phi_5(z),\,r=\Phi_6(z),
\ee
where the dependence on $(x,t)$ was omitted. 

Substitution of (\ref{3.4.3}) into (\ref{2.3.4}) reads
\bb\label{3.4.4}
\left\{\ba{lcl}
-c\Phi_{1}^{\prime}&=&\left(\Phi_{1}+\Phi_{2}\right)^p\Phi_{1}^{\prime\prime}  +p\,(\Phi_1+\Phi_2)^{p-1}(\Phi_1^{\prime}+\Phi_2^{\prime})\Phi_1^{\prime}  -2\nu(\Phi_1)^{q_1}\Phi_1^{\prime}
         +\frac{\gamma}{k}\Phi_{3}\\  
         &&+\left(\epsilon \frac{\gamma}{k}  -\mu_1\right)\Phi_1  -\beta_1\Phi_1\Phi_5,\\
       
-c\Phi_{2}^{\prime}&=&\left(\Phi_{1}+\Phi_{2}\right)^p\Phi_{2}^{\prime\prime}  +p\,(\Phi_1+\Phi_2)^{p-1}(\Phi_1^{\prime}+\Phi_2^{\prime})\Phi_2^{\prime}  
         -2\nu(\Phi_2)^{q_2}\Phi_2^{\prime}\\ & &
         +\left(\epsilon \frac{\gamma}{k}  -\mu_1\right)\Phi_2  +\beta_1\Phi_1\Phi_5,\\
        
-c\Phi_{3}^{\prime}&=&  k\,(\Phi_{1}+\Phi_{2})+(\epsilon k-\mu_2-\gamma)\Phi_{3},\\

-c\Phi_{4}^{\prime}&=& -\beta_2\Phi_2\Phi_4,\\
    
-c\Phi_{5}^{\prime}&=& \beta_2\Phi_2\Phi_4  -\sigma\Phi_5,\\

-c\Phi_{6}^{\prime}&=& \sigma\Phi_5.
\ea\right.      
\ee

\section{Spatial homogeneity}\label{spatial dynamics}

In what follows we make the assumption that $\epsilon=0$. This enables us to compare some of our results with those obtained in \cite{maidana2005}. It will be of great importance in our analysis the following quantities: the basic offspring number
\bb\label{4.1}
\Q=\f{\gamma}{\mu_1(\gamma+\mu_2)}
\ee
and the basic reproduction number
\bb\label{4.2}
{\cal R}_0=\f{\be_1\be_2h^\ast u^\ast}{\mu_1\sigma}.
\ee
The latter depends explicitly on the densities of susceptible mosquitoes $u^\ast$ and humans $h^\ast$. If we do not choose $\epsilon=0$, $\Q$ might depend on this parameter, which would not allow us to proceed to a comparison with the results of \cite{maidana2005}.

\subsection{Preliminaries}
At this point it will be useful to recall some facts on the theory of ordinary differential equations. To begin with, let $U\subseteq\R^n$,  $I\subseteq\R$, $u:I\rightarrow U$ and $f:U\rightarrow\R^m$ be, a connected open set, an interval, a smooth function such that $u(t_0)=u_0$, where $t_0\in I$, and a vector field, respectively. Consider the system of ordinary differential equations with initial condition
\bb\label{4.3}
\left\{\ba{l}
u'(t)=f(u),\\
\\
u(t_0)=u_0.
\ea\right.
\ee
If $u_0$ is such that $f(u_0)=0$, then $u\equiv u_0$ is said to be an \textit{\textit{equilibrium point}}  of the system (\ref{4.3}). In particular, it is also a solution of the system, called equilibrium solution.

A solution $u$ of the system (\ref{4.3}) is said to be stable if, for every $\eps>0$, there is a $\delta=\delta(\eps)$ such that for any other solution of (\ref{4.3}) $w$ for which $\|u-w\|<\delta$ at $t=t_0$, satisfies the further inequality $\|u-w\|<\eps$ for $t\geq t_0$. Otherwise, the solution $u$ is said to be unstable.

If $u_0$ is an equilibrium point of $f:U\rightarrow\R^m$, then $u_0$ is called asymptotically stable to $f$ if, for any $\eps>0$, there is $\delta>0$ such that, for any $u\in U$ and any $t>t_0$, $\phi_t(u)\rightarrow u_0$, when $t\rightarrow t_0$ and $\|u-u_0\|<\delta$ implies $\|\phi_t(u)-u_0\|<\eps$, where, for each $t$, $\phi_t(u)=\phi(t,u)$ and $\phi:\R\times U\rightarrow\R^m$ is the flux through $u$, that is, $\p_t\phi(t,u)=f(\phi(t,u))$.

In what follows, $f:U\subseteq\R^n\rightarrow\R^m$ is a vector field, $u_0\in U$ is a point, $J_f(u_0)$ is the Jacobian matrix of $f$ evaluated at $u_0$. The next three propositions can be found in \cite{doe}, pages 195, 198 and 49, respectively.

\begin{proposition}\label{prop1}
Let $u_0\in U$ be a point such that $f(u_0)=0$. If all eigenvalues of $J_f(u_0)$ have negative real part, then $u_0$ is an asymptotically stable point to $f$.
\end{proposition}

\begin{proposition}\label{prop2}
Let $u_0$ be an equilibrium point of $f$. If $J_f(u_0)$ has an eigenvalue with positive real part, then $u_0$ is an unstable point.
\end{proposition}

\begin{proposition}\label{prop3}
If $f$ is linear, $u_0$ is an equilibrium point of $f$ and all eigenvalues of $J_f(u_0)$ have real part negative or $0$, then $u_0$ is a stable point.
\end{proposition}

\subsection{Temporal dynamics}\label{homogeneity}

We now consider the temporal dynamics, or temporal conditons of the stability, of the system (\ref{2.3.4}). This condition is necessary for the spatial condition which determine the mosquitoes invasion.

Assuming that $\mu_3=0$ and $u_x=u_{xx}=w_x=w_{xx}=0$, system (\ref{2.3.3}) becomes 
\bb\label{4.1.1}
\left\{\ba{l}
\ds{\f{\text d u}{\text d t}=\f{\gamma}{k}v- \mu_{1}u- \beta_{1} u I,}\\
\\
   \displaystyle
   \f{\text d w}{\text d t}= -\mu_{1}w            + \beta_{1}u I,\\
   \\

   \displaystyle
   \f{\text d v}{\text d t}= k (u+w)
              -(\mu_{2}+\gamma)v,\\
              \\
  
   \displaystyle
   \f{\text d h}{\text d t}= -\beta_{2}hw,\\
\\
   \displaystyle
   \f{\text d I}{\text d t}=\beta_{2}hw
              -\sigma I,\\
              \\
   \displaystyle
   \f{\text d r}{\text d t}=\sigma I.
\ea\right.
\ee

The equilibrium points of the system (\ref{4.1.1}) are given by
    $w^*=0,\,\,
    I^*=0,\,\,
    r^*=1-h^*,\,\,
    0\leq h^*\leq1,$
and
\bb\label{4.1.3}
\ba{l}
  \ds{    u^*=\f{\gamma}{k\mu_1}v^*,}\,\,\,\,
  \ds{    v^*=\f{k}{\gamma+\mu_2}u^*.}
\ea
\ee

Equations (\ref{4.1.3}) can be equivalently rewritten as the system
\bb\label{4.1.4}
 A \begin{bmatrix}
  u^*\\
  v^*
 \end{bmatrix}
=\begin{bmatrix}
   0\\0
 \end{bmatrix},\qquad
  A=
  \begin{bmatrix}
  -1                         &{\gamma}/{(k\mu_1)}\\
  {k}/{(\gamma+\mu_2)}& -1
\end{bmatrix}.
\ee

A quick calculation shows that $\det(A)=1-\f{\gamma}{(\gamma+\mu_2)\mu_1}=1-\Q.$

In view of the latter equation, system (\ref{4.1.4}) will have unique solution if and only if $\det(A)\neq0$, which is equivalent to $\Q\neq1$. Provided that such a condition holds, we have the following set of equilibrium points:
\bb\label{4.1.5}
    \text{E}_0=\{(u^*,w^*,v^*,h^*,I^*,r^*)=(0,0,0,h^*,0,1-h^*),\quad 0\leq h^*\leq1\}.
\ee

On the other hand, if one assumes that $\Q=1$, then the system (\ref{4.1.4}) loses the uniqueness of solutions. Consequently, the equilibrium points of (\ref{4.1.1}) belong to the region
\bb\label{4.1.6}
    \text{E}_1=\left\{\left(u^*,w^*,v^*,h^*,I^*,r^*\right)=\left(\f{v^\ast\gamma}{k\mu_1},0,v^\ast,h^*,0,1-h^*\right),\quad 0\leq h^*\leq1,\quad v^\ast>0\right\}.
\ee

\begin{remark}
If we had not assumed $\epsilon=0$, then the system $(\ref{4.1.4})$ would have been
\bb\label{4.1.7}
A_\epsilon \begin{bmatrix}
  u^*\\
  v^*       
   \end{bmatrix}     
=\begin{bmatrix}
   0\\0
 \end{bmatrix},\qquad
  A_\epsilon= \begin{bmatrix}
-1           &\f{\gamma}{k\mu_1-\epsilon\gamma}\\
\f{k}{\mu_2+\gamma-\epsilon k}&  -1
\end{bmatrix}
\ee
and $\det(A_\epsilon)=1-\f{k\gamma}{(k\mu_1-\epsilon\gamma)(\gamma+\mu_2-\epsilon k)}=:1-{\cal Q}_\epsilon.$
Then the values of the basic offspring number depend on $\epsilon$ and, in order to compare with the results of \cite{maidana2005}, we must take $\epsilon=0$.

\end{remark}

\subsection{Jacobian matrices and eigenvalues in the dynamics of humans and mosquitoes}

Let $p=(0,0,0,h^*,0,1-h^*)\in  \text{E}_0$. The Jacobian matrix associated to the system (\ref{4.1.1}) at $p$ is given by
\bb\label{4.2.1}
  \text J(p)=\left[
    \begin{array}{cccccc}
      -\mu _1 & 0 & \gamma/k & 0 & 0 & 0 \\
      0 & -\mu _1 & 0 & 0 & 0 & 0 \\
      k & k & -(\gamma +\mu _2) & 0 & 0 & 0 \\
      0 & -\beta _2 h^* & 0 & 0 & 0 & 0 \\
      0 & \beta _2 h^* & 0 & 0 & -\sigma  & 0 \\
      0 & 0 & 0 & 0 & \sigma  & 0 \\
    \end{array}\right].
\ee

The eigenvalues of (\ref{4.2.1}) are $0$ and the roots of the polynomial
\bb\label{4.2.2}
 \text P_1(\lambda)=
  \lambda ^2
 +\lambda   (\gamma +\mu _1+\mu _2)
 +\mu_1(1-\Q)(\gamma +\mu _2),
\ee
are
\bb\label{4.2.3}
  \lambda_{1,2} = \frac{1}{2}\left[-(\gamma +\mu _1+\mu _2)\pm\sqrt{\left(\gamma +\mu _1+\mu _2\right){}^2 + 4 \mu _1 (\Q-1)\left(\gamma +\mu _2\right)}\right].
\ee

On the other hand, the Jacobian matrix associated to the system (\ref{4.1.1}) evaluated at a point $q=(v^\ast\gamma/(k\mu_1),0,v^\ast,h^*,0,1-h^*)\in \text{E}_1$ is
{\small\bb\label{4.2.4}
    \text J(q)=
    \left[
      \begin{array}{cccccc}
        -\mu _1 & 0 & \gamma/k & 0 & -{\gamma  \beta _1 v^*}/{k \mu _1} & 0 \\
        0 & -\mu _1 & 0 & 0 & {\gamma  \beta _1 v^*}/{k \mu _1} & 0 \\
        k & k & -(\gamma +\mu _2) & 0 & 0 & 0 \\
        0 & -\beta _2 h^* & 0 & 0 & 0 & 0 \\
        0 & \beta _2 h^* & 0 & 0 & -\sigma  & 0 \\
        0 & 0 & 0 & 0 & \sigma  & 0 \\
      \end{array}
    \right],
\ee}
whose eigenvalues are $0$ and, in addition, the roots of the polynomials 
\bb\label{4.2.5}
\ba{l}
    \text P_2(\lambda) =\lambda ^2
   +\lambda \left(\gamma +\mu _1+\mu _2\right)
   +\mu _1 (1-\Q) \left(\gamma +\mu _2\right)
   =\lambda(\lambda+(\gamma +\mu _1+\mu _2)),\\
   \\
    \text P_3(\lambda)=
    \lambda ^2
    +\lambda\left(\mu _1+\sigma \right)
    +\mu _1 \sigma(1-\Ro).
\ea
\ee
In $\text P_2(\lambda)$ above, we made use the fact that $\Q=1$ in the region $\text{E}_1$.

Our first results regarding the equilibrium points can now be announced.

\begin{theorem}\label{teo1}
Let $p\in\text{E}_0$ and $\Q$, where $\text{E}_0$ and $\Q$ are given in $(\ref{4.1.5})$ and $(\ref{4.1})$, respectively. If ${\cal Q}_0>1$, then the equilibrium point $p$ is unstable.
\end{theorem}

\begin{proof}
Let $p\in\text{E}_0$ and $J(p)$ given by (\ref{4.2.1}). Assuming $\Q>1$, it follows from (\ref{4.2.3}) that $\lambda_1>0$. Thus, by Proposition \ref{prop2}, we have proved the theorem in question.
\end{proof}

\begin{theorem}\label{teo2}
Let $\Q=1$ and $q\in\text{E}_1$ and $\Ro$, where $\text{E}_1$ and $\Ro$ are given in $(\ref{4.1.6})$ and $(\ref{4.2})$, respectively. If ${\cal R}_0>1$, then the equilibrium point $q$ is unstable.
\end{theorem}

\begin{proof}
Consider the polinomial $\text P_3(\lambda)$ given by (\ref{4.2.5}). A straightforward calculation shows that 
$$\lambda= \frac{1}{2}\left[-(\sigma +\mu _1)\pm\sqrt{(\sigma +\mu _1)^2+4\sigma\mu _1(\Ro-1)}\right]>0$$
under the hypothesis of the theorem. This implies that the matrix (\ref{4.2.4}) has at least one eigenvalue with positive real part and again, the result is a consequence of Proposition \ref{prop2}. 
\end{proof}

\subsection{Analysis of the mosquitoes population}

Seeing that in both equilibrium regions $\text{E}_0$ and $\text{E}_1$ we have $I^\ast=0$, which implies in the absence of infected humans, the first three equations of (\ref{4.1.1}) have a dynamics independent of that of humans. Moreover, if $I=0$, then $w=0$ and we shall therefore pay attention to the following subsystem of (\ref{4.1.1}):
\bb\label{4.3.1}
\left\{\ba{l}
   \displaystyle
   \f{\text d u}{\text d t}=
              \f{\gamma}{k}v
            - \mu_{1}u,\\
         \\   
   \displaystyle
   \f{\text d v}{\text d t}= k u
              -(\mu_{2}+\gamma)v.
              \ea\right.
\ee

Proceeding as in the previous sections, we have the following set of equilibrium points (restricting to a bidimensional space with coordinates $(u,v)$):
\bb\label{4.3.2}
E'=\{(\gamma v^\ast/(k\mu_1),v^\ast),\quad v^\ast\geq0\}.
\ee

A featured point of $E'$ is $\text e_0=(0,0)$, which can only be achieved provided that $\Q\neq 1$. Otherwise, if $\Q=1$ and $ \text e\in E'$, then $ e\neq \text e_0$. In what follows, we denote by $\text e$ any point of $E'$ different from $\text e_0$.

The Jacobian associated to (\ref{4.3.1}) evaluated at $\text e_0$ and $ \text e\in E'$ are, respectively, given by
\bb\label{4.3.3}
\text J(\text e_0)=
\text J(\text e)=
\lr[]{
  \begin{array}{ccc}
    -\mu _1 &  {\gamma }/{k} \\
    k &  -(\gamma +\mu _2) \\
  \end{array}},
\ee
whose characteristic polynomial and eigenvalues are, respectively,
\bb\label{4.3.4}
  \bar{\text P}(\lambda)=\lambda ^2+\lambda  \left(\gamma +\mu _1+\mu _2\right)-\mu_1(\gamma+\mu_1)(\Q-1),
\ee

\bb\label{4.3.5}
  \lambda_{1,2}^*=\f{1}{2}\lr[]{-(\gamma+\mu_1+\mu_2)\pm \sqrt{(\gamma+\mu_1+\mu_2)^2+4\mu_1(\gamma+\mu_1)(\Q-1)}}.
\ee

Now we present the main results concerning qualitative aspects of system (\ref{4.3.1}). We begin with two auxiliary lemmas.

\begin{lemma}\label{lema1}
If $\Q<1$, then the real part of the roots $\lambda_{1,2}^\ast$ $(\ref{4.3.5})$ are negative.
\end{lemma}

\begin{proof}
By the Routh--Hurwitz criteria (see \cite{murray}, Appendix B), the real part of the roots of the polinomial (\ref{4.3.4}) are negative if all coefficients of the polinomial are positive. Once $\gamma$, $\mu_1$ and $\mu_2$ are positive, then $\gamma +\mu _1+\mu _2>0$. Since $\Q\in(0,1)$, then $\mu_1(\gamma+\mu_1)(1-Q_0)>0$ and the results follows.
\end{proof}

\begin{lemma}\label{lema2}
If $\Q>1$, then the roots $\lambda_{1,2}^\ast$ $(\ref{4.3.5})$ are non-zero and have opposite signs.
\end{lemma}
\begin{proof}
If $\Q>1$, then $\sqrt{(\gamma+\mu_1+\mu_2)^2+4\mu_1(\gamma+\mu_1)(\Q-1)}>(\gamma+\mu_1+\mu_2)$. This inequality implies
$$
\ba{l}
      \lambda_1^*=\f{1}{2}\lr[]{-(\gamma+\mu_1+\mu_2)- \sqrt{(\gamma+\mu_1+\mu_2)^2+4\mu_1(\gamma+\mu_1)(\Q-1)}}<-2(\gamma+\mu_1+\mu_2)<0,\\
      \\
      \lambda_2^*=\f{1}{2}\lr[]{-(\gamma+\mu_1+\mu_2)+ \sqrt{(\gamma+\mu_1+\mu_2)^2+4\mu_1(\gamma+\mu_1)(\Q-1)}}>0.
\ea
$$
\end{proof}

\begin{theorem}\label{teo3}
The equilibrium point $\text e_0=(0,0)$ of the system $(\ref{4.3.1})$ is asymptotically stable if $\Q<1$ and unstable if $\Q>1$.
\end{theorem}

\begin{proof}
If $\Q<1$, it follows from Lemma \ref{lema1} that all eigenvalues (\ref{4.3.3}) have negative real part and the asymptotic stability follows from Proposition \ref{prop1}.

On the other hand, if $\Q>1$, by Lemma \ref{lema2} we have $\lambda_2^\ast>0$, which implies that (\ref{4.3.3}) has an eigenvalue with positive real part. From Proposition \ref{prop2}, $\text e_0$ is an unstable point of (\ref{4.3.1}).
\end{proof}

\begin{theorem}\label{teo4}
Let $\text e\in E'$, where $E'$ is given by $(\ref{4.3.2})$. Then $\text e$ is an stable equilibrium point of the system  $(\ref{4.3.1})$.
\end{theorem}

\begin{proof}
If $\Q=1$, then the eigenvalues $\lambda_{1,2}^*$ of (\ref{4.3.3}) are given by $\lambda_{1}^*=0 $ and $\lambda_{2}^*=-(\gamma+\mu_1+\mu_2)<0$. Then the result follows from Proposition \ref{prop3}.
\end{proof}

{ \tiny
\begin{figure}[!h]
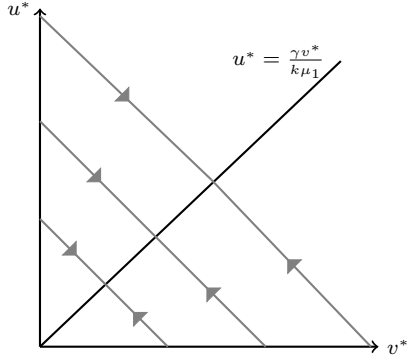

  \centering
 \figEum
 \caption{\scriptsize An illustration of the stability in the bidimensional case for $\Q=1$ (Theorem \ref{teo4}).}
 \label{fig:e1}
\end{figure}}

\section{Wave speed during the spatial mosquitoes' invasion}\label{speed}

Here we determine the wave speed during the spatial invasion of mosquitoes, which we would like to compare with analogous results obtained in \cite{maidana2008}. For this reason, in addition to the hypothesis $\epsilon=0$, we shall not consider nonlinearities in the diffusion as well as in the advection, that is, we consider $p=q_1=q_2=0$ in (\ref{2.3.4}). It will be of great importance in our analysis the basic offspring number (\ref{4.1}) and the basic reproduction rate (\ref{4.2}). In dimensional variables, they are given by

\bb\label{numbers}
\Q=\f{\bar{\gamma}}{\bar{\gamma}+\bar{\mu}_2}\times\f{\bar{r}_0}{\bar{\mu}_1},\quad\Ro=\f{\bar{\be}_1 Nh^\ast}{\bar{\mu}_1}\times\f{\bar{\beta}_2k_1u^\ast}{\bar{\sigma}}.
\ee

From the hypothesis on $p$, $q_1$, $q_2$ and $\epsilon$, and the fact that we shall determine the wave speed during the mosquitoes' invasion, it is sufficient to study system (\ref{3.4.4}) with these conditions. 

\subsection{Equilibrium points of system (\ref{3.4.4}) with $p=q_1=q_2=\epsilon=0$}

Defining auxiliary functions $\Psi_1(z):=\Phi_1'(z)$, $\Psi_2(z):=\Phi_2(z)$, system (\ref{3.4.4}) can be transformed into
\bb\label{6.1}
\left\{\ba{l}
    \Phi_{1}^\prime=\Psi_1,\\
    
    \ds{\Psi_{1}^\prime=
      (2\nu-c)\Psi_1
      -\frac{\gamma}{k}\Phi_{3}
      +\mu_1\Phi_1    +\beta_1\Phi_1\Phi_5},\\
      
    \Phi_2^\prime=\Psi_2,\\
    
    \Psi_2^\prime=
      (2\nu-c)\Psi_2
      +\mu_1\Phi_2    -\beta_1\Phi_1\Phi_5\\
      
    \ds{\Phi_3^\prime=
         -\f{k}{c}(\Phi_{1}+\Phi_{2})+\f{(\mu_2+\gamma)}{c}\Phi_{3}},\\
         
    \ds{\Phi_4^\prime=
        \f{\beta_2}{c}\,\Phi_2\Phi_4},\\
        
   \ds{ \Phi_5^\prime=
       -\f{\beta_2}{c}\,\Phi_2\Phi_4  +\f{\sigma}{c}\Phi_5},\\
       
    \ds{\Phi_6^\prime=\,\,\,\,\,-
        \f{\sigma}{c}\Phi_5.}
\ea\right.
\ee

Recalling that $h+I+r=1$, the set of equilibrium points of (\ref{6.1}) is given by 
\bb\label{6.2}
\hat{\text{E}}_0=\{(\Phi_1^*, \Psi_1^*, \Phi_2^*, \Psi_2^*, \Phi_3^*, \Phi_4^*, \Phi_5^*, \bar \Phi_6^*)=(0,0,0,0,0,h^\ast,0,1-h^\ast),\quad h^\ast\in[0,1]\}
\ee
if $\Q\neq1$, and by
\bb\label{6.3}
\hat{\text{E}}_1=\left\{\left(\Phi_1^*, \Psi_1^*, \Phi_2^*, \Psi_2^*, \Phi_3^*, \Phi_4^*, \Phi_5^*, \bar \Phi_6^*\right)=\left(0,\f{v^*\gamma}{k\mu_1},0,0,v^\ast,h^\ast,0,1-h^\ast\right),\,\, v^\ast>0,\,\, h^\ast\in[0,1]\right\}
\ee
provided that $\Q=1$.

The Jacobian associated to (\ref{6.1}) at a point $\hat{p}\in\hat{\text{E}}_0$ and a point $\hat{q}\in\hat{\text{E}}_1$ are, respectively,
\bb\label{6.4}
    \text J(\hat{p})=
    \left[
      \begin{array}{cccccccc}
        0 & 1 & 0 & 0 & 0 & 0 & 0 & 0 \\
        \mu _1 & 2 \nu-c  & 0 & 0 & -\gamma/k & 0 & 0 & 0 \\
        0 & 0 & 0 & 1 & 0 & 0 & 0 & 0 \\
        0 & 0 & \mu _1 & 2 \nu-c  & 0 & 0 & 0 & 0 \\
        k/c & 0 & k/c & 0 & -(\gamma +\mu _2)/c & 0 & 0 & 0 \\
        0 & 0 & -(\beta _2 h^*)/c & 0 & 0 & 0 & 0 & 0 \\
        0 & 0 & (\beta _2 h^*)/c & 0 & 0 & 0 & \sigma/c & 0 \\
        0 & 0 & 0 & 0 & 0 & 0 & \sigma/c & 0 \\
      \end{array}\right]
\ee
and
\bb\label{6.5}
    \text J(\hat{q})=
    \left[
      \begin{array}{cccccccc}
        0 & 1 & 0 & 0 & 0 & 0 & 0 & 0 \\
        \mu _1 & 2 \nu-c  & 0 & 0 & -\gamma/k & 0 & (\gamma  \beta _1 v^*)/k \mu _1 & 0 \\
        0 & 0 & 0 & 1 & 0 & 0 & 0 & 0 \\
        0 & 0 & \mu _1 & 2 \nu-c  & 0 & 0 & -(\gamma  \beta _1 v^*)/k \mu _1 & 0 \\
        k/c & 0 & k/c & 0 & -(\gamma +\mu _2)/c & 0 & 0 & 0 \\
        0 & 0 & -\beta _2 h^*/c & 0 & 0 & 0 & 0 & 0 \\
        0 & 0 & \beta _2 h^*/c & 0 & 0 & 0 & \sigma/c & 0 \\
        0 & 0 & 0 & 0 & 0 & 0 & \sigma/c & 0 \\
      \end{array}
    \right].
\ee

\subsection{Method for determining the wave speed}

Our procedure for determining  the wave speed follows closely that employed in \cite{maidana2008,maidana2005} for determining $c$: let $J$ be (\ref{6.4}) or (\ref{6.5}). Denoting by $\text P(\lambda,c)$ the corresponding characteristic polynomial, we determine the critical points of $\text P(\cdot,c)$. Since both (\ref{6.4}) and (\ref{6.5}) have $2$ columns with all entries $0$ it follows that the characteristic polynomial can be factored into two polynomials of degree three. In order to assure a third order degree polynomial $ \text P(\lambda,c)$ has only real roots, one must impose that at least one of the roots of $ \text P(\lambda,c)$ is negative and the constant $c$ must be chosen such that $ \text P(\lambda,c)$ has at least one positive root.

So, in order to achieve the aforementioned requirements, one should impose the following conditions:
\bb\label{6.6}
    \text P(0,c)>0,\qquad\qquad\lim\limits_{\lambda\rightarrow\pm\infty}  \text P(\lambda,c)=\pm\infty.
\ee

To assure the existence of $c$ such that $ \text P(\lambda,c)$ has at least one positive real root, one must impose that for a fixed $c> \text C_{\text{min}}$, the following conditions holds
\bb\label{6.7}
    \lambda_+>0,\qquad\qquad\left.\f{\text d\text P(\lambda,c)}{\text d\lambda}\right|_{\lambda=0}<0,
\ee
where 
\bb\label{6.8}
\lambda_+=\max\left\{\lambda:\ \f{\p\,\text P(\lambda,c)}{\p \lambda}=0\right\},\,\,
 \,\,
 \text C_{\text{min}}=\left\{c:\ \text P(\lambda_{\text{+}},c)=0\right\}.
\ee

\subsection{Wave speed for mosquitoes' invasion}

The eigenvalues of (\ref{6.4}) are $\lambda_1=\lambda_2=0$, $\lambda_3=-\sigma/c$ and roots of the following polynomials:
\bb\label{6.2.1}
\ba{l}
\hat{\text P}_0(\lambda,c)=\lambda ^2-\lambda (2 \nu-c )-\mu _1,\\
\\
\hat{\text P}_1(\lambda,c)=\ds{\lambda ^3
     + \lambda ^2\lr(){{(c-2 \nu) +\f{\gamma +\mu _2}{c}}}
     - \lambda   \lr(){\frac{(2 \nu -c)\left(\gamma +\mu _2\right)}{c}+\mu _1}
     + \frac{\mu _1 \left(\gamma +\mu _2\right)(Q_0-1)}{c}},
     \ea
\ee
where $Q_0$ is given by (\ref{4.1}).

Once the discriminant of $\hat{\text P}_0(\lambda,c)$ in (\ref{6.2.1}) satisfies $(2 \nu-c )^2+4\mu_1>0$ for any values of $c$, $\nu$ and $\mu_1$, this implies that its roots are always real numbers.

With respect to $\hat{\text P}_1(\lambda,c)$ in (\ref{6.2.1}), assuming that $\Q>1$, a condition already obtained  in the previous section in order to describe the mosquitoes' invasion, we have $\hat{\text P}_1(0,c)=[\mu _1 \left(\gamma +\mu _2\right)(Q_0-1)]/c>0$ for any $c>0$. Moreover, we have $\hat{\text P}_1(\lambda,c)\rightarrow\pm\infty$ when $\lambda\rightarrow\pm\infty$.

It follows from (\ref{6.2.1}) that the critical points of $\hat{\text P}_1(\cdot,c)$ are
\[\lambda_\pm=\f{1}{3}\lr[]{-\lr(){-(2 \nu-c) +\f{\gamma +\mu _2}{c}}\pm
              \sqrt{\lr(){-(2 \nu-c) +\f{\gamma +\mu _2}{c}}^2+3\lr(){\frac{\left(\gamma +\mu _2\right) (2 \nu-c)}{c}+\mu _1}}},\]
              
It is reasonable to assume that the traveling wave should prevail on the wind if it is opposite to it, that is, one should impose that  $2\nu-c>0$, from which we conclude that $\lambda_+>0$ and
$$\lr.|{\f{\text d \hat{\text{P}}_1(\lambda,c)}{\text d\lambda}}_{\lambda=0}=-\lr(){\f{\left(\gamma +\mu _2\right) (2 \nu -c)}{c}+\mu _1}<0.$$
The polynomial $\hat{\text{P}}_1$ satisfies the conditions of the method employed to find out the wave speed. Hence, we use this polynomial to find the wave speed of mosquitoes' invasion. The procedure to obtain the minimum wave speed is illustrated in Figure \ref{fig2}.

\begin{table}[!h]\centering
 \scalefont{0.8} \caption{\scalefont{0.9} Biological parameters used in the simulations for finding the speed velocity of the wave. These values are taken from the reference \cite{maidana2008}. The last 6 parameters are taken considering a temperature at 15\Celsius\, and 30\Celsius. The values used for calculating the basic offspring number and the basic reproduction rate (see (\ref{numbers})) are $u^*=v^*\gamma/(k\mu_1)$, $v^*=0.7$ and $h^*=1$.}
    \begin{tabular}[!h]{|l|c|l|} \hline
      {Parameter} & {Value} \\\hline
    Diffusion coefficient $\bar D$ & 1.25 $\times$ 10\tSup{-2} km\tSup2/day\\
    Advection coefficient  $2\bar \nu$&5 $\times$ 10\tSup{-2} km/day\\
    Carrying capacity (winged) $k_1$& 25 individuals/km\tSup2\\
    Carrying capacity (aquatic)  $k_2$& 100 individuals/km\tSup2\\
     Transmission coefficient (humans$\rightarrow$ mosquitoes) $\bar{\beta}_1$& 0.0033 { km\tSup{2}/day}\\
    Transmission coefficients  (mosquitoes$\rightarrow$ humans)$\bar{\beta}_2$& 0.0025 { km\tSup{2}/day}\\
     Period of infection $\phantom{^{-1}}(\bar{\sigma})^{-1}$ &  {7 days}\\
     Number of humans $\bar{N}$ & 150 { individuals/km\tSup{2}}\\
     Rate of oviposition $\bar r_0$         & 1.52 (15\Celsius)/10  (30\Celsius) days$^{-1}$\\
     Average time in the aquatic form $\phantom{^{-1}}(\bar\gamma)^{-1}$  & 52.63 (15\Celsius)/5 (30\Celsius) days\\
     Average lifetime in the winged form $\phantom{^{-1}}(\bar\mu_1)^{-1}$    & 26.3 (15\Celsius)/ 35 (30\Celsius) days\\
     Average lifetime in the aquatic form $\phantom{^{-1}}(\bar\mu_2)^{-1}$    & 50 (15\Celsius)/18 (30\Celsius) days\\
     Basic offspring number $Q_0$                                                & 19.45 (15\Celsius)/273.91 (30\Celsius)\\
     Basic reproduction rate $\Ro$                                                 & 7.97 (15\Celsius)/148.46 (30\Celsius)\\\hline
  \end{tabular}
  \label{tab3}
\end{table}

Let $p\in \hat{\text{E}}_0$ (see Eq. (\ref{6.2})). From the values given in Table \ref{tab3} and considering a wind current with velocity $2\bar\nu\,=\,18.25\, km/year$ we obtained the speed $\bar c=\, 89.67 \,km/year$. In the case of absence of wind 
$2\bar\nu=0$, it is obtained $\bar c\,=\,75.46\, km/year$.

\begin{figure}[!ht]
  \centering
  \includegraphics[scale=0.70]{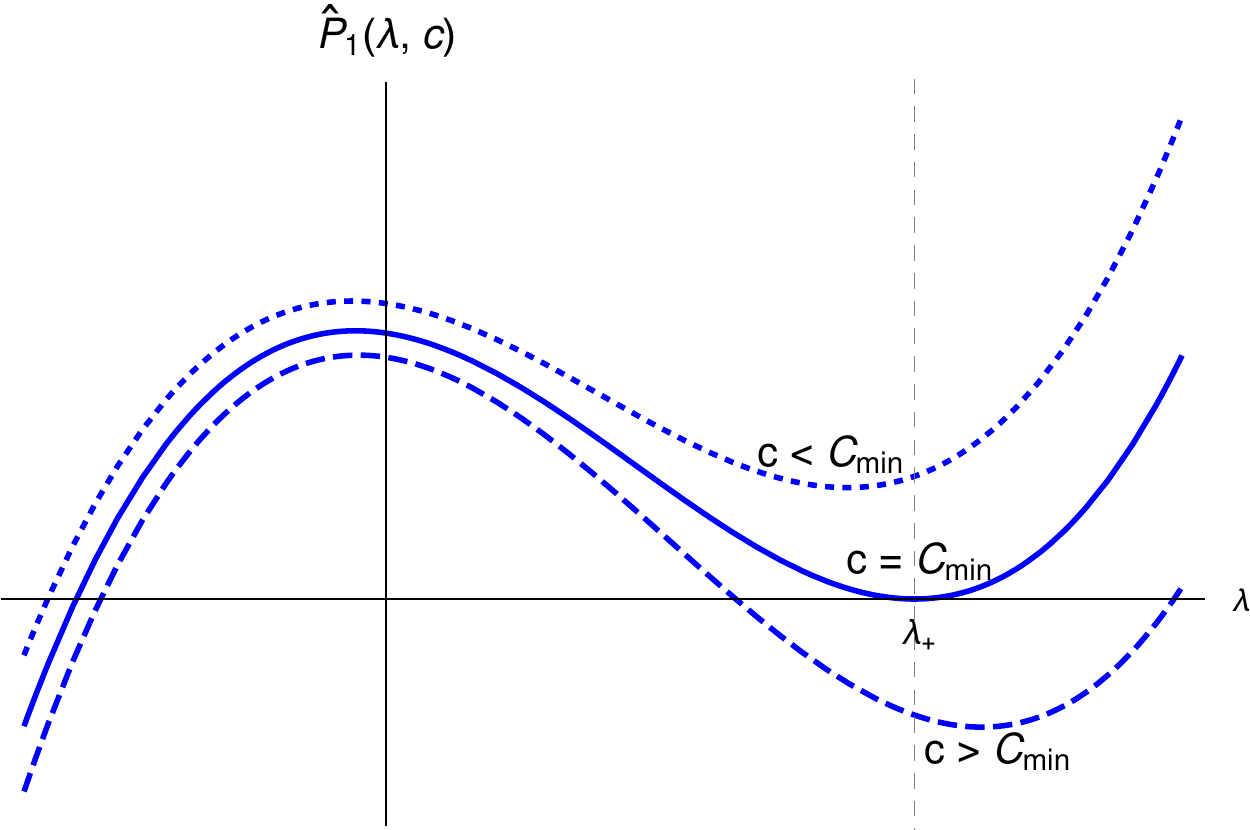}
  \caption{\scalefont{0.9}The figure shows the result of the procedure above for obtaining the minimum wave speed for the wave. Using the values of Table \ref{tab3} it was found the non-dimensional value $\text C_{\text{min}}=0.69$ for the polynomial $\hat{\text P}_1(\lambda,c)$ given in (\ref{6.2.1}), corresponding to an equilibrium point $\hat{p}\in \hat{\text{E}}_0$ (\ref{6.2}).}
  \label{fig2}
\end{figure}

\subsection{Wave propagation for $\Q=1$ and $\Ro>1$: dengue's dispersion}

Let $\hat{q}\in\hat{\text{E}}_1$. The eigenvalues of the matrix (\ref{6.5}) are $\lambda_1=\lambda_2=\lambda_3=0$ and the roots of the polynomials
\bb\label{5.4.1}
\ba{l}
\ds{\hat{\text P}_2(\lambda,c)=\lambda ^2
      +\lambda\left(-(2 \nu-c)+\f{\gamma+\mu_2}{c}\right)
      -\left(\f{(2 \nu-c)(\gamma+\mu_2)}{c} + \mu _1\right)},\\
      \\
\ds{\hat{\text P}_3(\lambda,c)=\lambda ^3
      +\lambda ^2 \left(-(2 \nu-c)+\f{\sigma }{c}\right)
      -\lambda\lr(){\f{\sigma}{c}(2\nu-c)+\mu _1}
      +\f{\sigma\mu_1(\Ro-1)}{c}.}
\ea
\ee
Proceeding as in the last subsection, we intend to determine restrictions on the biological parameters involved in (\ref{numbers}) in order to have $\Q=1$ and $\Ro>1$.

Taking (\ref{numbers}) into account and since $u^*=v^*\gamma/(k\mu_1)$, we have
\bb\label{5.4.2}
\Ro=\f{\bar{\be}_1 Nh^\ast}{\bar{\mu}_1}\times\f{\bar{\beta}_2}{\bar{\sigma}}\times\f{\bar{v}^\ast\bar{\gamma}}{\bar{\mu}_1}.
\ee

According to Table \ref{tab3}, the lowest possible value to $\Q$ is achieved at 15\Celsius. Even for this choice, maintained the values of the biological parameters, $\Q$ is significantly greater than 1. A natural way for decreasing $\Q$ without affecting the basic reproduction rates (\ref{5.4.2}) would be increasing the value of $\bar{\mu}_2$. Then, considering the values of $\bar{\gamma}^{-1}$, $\bar{\mu}_1^{-1}$ and $\bar{r}_0$ given on Table \ref{tab3} and imposing that $\Q=1$, one finds 
\bb\label{5.4.3}
\bar{\mu}_2=0.74.
\ee 

Fixing $v^\ast=0.7$ and $h^\ast=1$, if $\bar{\nu}=0$ we would obtain $\bar{c}=24.08\,km/year$, while if $2\bar{\nu}=18.25\,km/year$, we would get $\bar{c}=38.72\,km/year$.

Figure \ref{fig3} shows the distribution of velocities $c_{min}$ as a function of $v^\ast$. We can observe an increasing of the wave speed when the densities of mosquitoes population increases.

\begin{figure}[!htb]
\centering
\includegraphics[scale=0.70]{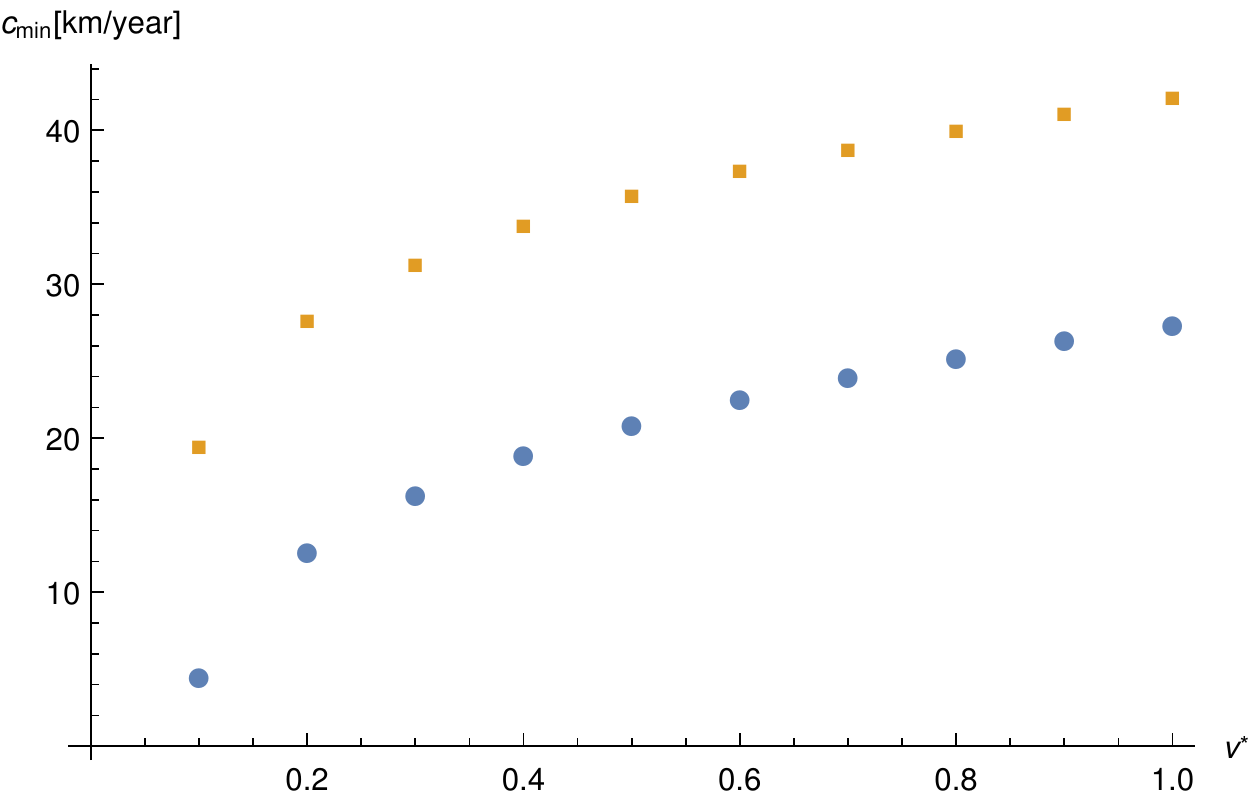}
\caption{\scalefont{0.9}Values of $c_{min}$ depending on $v^\ast$. The points represented by circles are obtained considering $\bar{\nu}=0$, while those represented by squares are calculated using $2\bar{\nu}=18.25\,km/year$. The values of $v^\ast$ are taken equally distributed from 0.1 to 1.0. The corresponding values of $c_{min}$ are 4.6026, 12.7056, 16.4107, 18.9825, 20.9895, 22.6519, 24.0797, 25.3367, 26.463, 27.4859, with $\bar{\nu}=0$. If $2\bar{\nu}=18.25\,km/year$, then the values of $c_{min}$ are 19.3765, 27.5711, 31.2148, 33.7331, 35.6967, 37.3231, 38.7206, 39.9514, 41.055, 42.0577.}
\label{fig3}
\end{figure}
\section{Discussion and Conclusion}\label{conclusion}

In this paper we derived two Malthusian models for analysing the transmission of dengue between humans and mosquitoes. These models can be viewed as members of the system (\ref{2.3.4}), and some Lie symmetries are listed on Table \ref{tab2}. 

Our results on symmetry analysis show that the transmission coefficients from human to mosquitoes $\bar{\be}_1$ and mosquitoes to humans $\bar{\be}_2$ and the wind current $2\bar{\nu}$, are quite relevant in the manifestation of symmetries other than the translations.

With respect to the power nonlinearities, the most dominant from the point of view of symmetries is $p$. The powers $q_1$ and $q_2$ are relatively important while $\nu$, that is related to the existence of wind currents, seems to be more relevant than $q_1$ and $q_2$. The most important biological parameters in this analysis are $\be_1$ and $\be_2$, which are related to the transmission between humans to mosquitoes and vice-versa.

The temporal dynamics of system (\ref{4.1.1}) shows the existence of regions of instability. They are given by the sets (\ref{4.1.5}) and (\ref{4.1.6}) provided that $\Q>1$ or $\Q=1$ and $\Ro>1$, respectively, as proved in theorems \ref{teo1} and \ref{teo2}. The dynamics of the mosquitoes population is also analysed. For this case Theorem \ref{teo3} shows that the origin, provided that $\Q<1$, is a point of asymptotic stability in the bidimensional space $(u,v)$, corresponding to the winged and aquatic forms of the mosquitoes. For $\Q>1$, the origin is unstable. Apart from the origin, in the bifurcation case $\Q=1$, all points of the set given in $(\ref{4.3.2})$ are stable, as proved in Theorem \ref{teo4}.

The  condition $\Q < 1$ leads to the eradication of the mosquitoes' population. On the other hand, $\Q  > 1$ means the invasion of the mosquitoes' population and, since the model is Malthusian, the growth of the population is unlimited. The value $\Q=1$ corresponds to a bifurcation value and a region of nontrivial points for the mosquitoes density is possible. In this case, when $\Ro>1$, the propagation of dengue disease could be possible. 

A case of biological relevance occurs when $\bar{\be}_1\bar{\be}_2\bar{\sigma}\neq0$. Under these restrictions, we have transmission of the virus among all populations and we also have recovering among humans. Using the principal Lie Algebra $L_{\cal P}$ we construct the invariant $z=x-ct$ and transforms system (\ref{2.3.3}), with $p=q_1=q_2=\mu_3=0$, into the system of second order ODEs (\ref{3.4.4}). For its own turn, this system can be transformed into a system of first order ODEs given by (\ref{6.1}). From the analysis of the linear part of the system (\ref{6.1}) and using the data from \cite{maidana2008} we determined that the wave of minimum speed has velocity $c=89.67\,km/year$, considering the biological data at 30\Celsius\, and a wind current with velocity of $18.25\,km/year$. This is the same result obtained in \cite{maidana2008}. The value of the wave speed can reach to $75.46\,km/year$, in the absence of wind currents. Again, a result in agreement with that obtained previously in \cite{maidana2008}.

On the other hand, if $\Q=1$ and $\Ro>1$, we obtain the speed of the spatial dispersion of the dengue as a function of the mosquitoes density as shown in Figure \ref{fig3}. However, in order to have this situation we should have a mortality rate given by (\ref{5.4.3}), which seems to be unrealistic. Biologically speaking, the situation $\Q=1$ would correspond to a high mosquitoes mortality. From mathematical viewpoint, condition $\Q=1$ implies bifurcation points, which brings changes in the stability of the system (\ref{6.1}) and hence, hardly describes a real situation.
 
\section*{Acknowledgements}

The authors are grateful to FAPESP, grant nº 2014/05024-8, for financial support. F. Bacani is thankful to CNPq, grant nº 141081/2014-7, for the scholarship provided. I. L. Freire is grateful to CNPq for the grant nº  308941/2013-6. M. Torrisi has been supported by Gruppo Nazionale per la Fisica Matematica of
Instituto Nazionale di Alta Matematica (Italy)


\begin{thebibliography}{99}

\bibitem{ai} J-W Ai, Y. Zhang and W. Zhang, Zika virus outbreak: `a perfect storm', Emerging Microbes and Infections, (2016), DOI: 10.1038/emi.2016.42.

\bibitem{felipo-tese} F. Bacani, Tratamento de modelos para a din\^amica populacional do Aedes aegypti via simetrias de Lie, PhD thesis in Applied Mathematics, State University of Campinas, (2016) -- in Portuguese.


\bibitem{bk} G. W. Bluman and S. Kumei, Symmetries and Differential Equations,
 Applied Mathematical Sciences 81, Springer, New York, (1989).
 
 \bibitem{butler} D. Butler, Zika and birth defects: what we know and what we don?t, Nature News, (2016), available at http://www.nature.com/news/zika-and-birth-defects-what-we-know-and-what-we-don-t-1.19596. Access made on June 15th 2016. 
 
 \bibitem{dimas1} S. Dimas, D. Tsoubelis, SYM: A new symmetry-finding package for Mathematica, in: Proceedings of the 10th International Conference in Modern Group Analysis, Larnaca, Cyprus, 24--30 October 2004, 2004, pp. 64--70.

\bibitem{dimas2} S. Dimas, D. Tsoubelis, A new heuristic algorithm for solving overdetermined systems of PDEs in Mathematica, in: 6th International Conference on Symmetry in Nonlinear Mathematical Physics, Kiev, Ukraine, 20--26 June 2005, 2005.

\bibitem{doe} C. I. Doering and A. O. Lopes, Equa\c{c}\~oes diferenciais ordin\'arias, IMPA, (2010) -- in Portuguese.


\bibitem{frema2013} I. L. Freire and M. Torrisi, Symmetry methods in mathematical modeling {\it Aedes aegypti} dispersal dynamics, Nonlin. Anal. RWA, vol. 14, 1300--1307, (2013).

\bibitem{frema2014} I. L. Freire and M. Torrisi, Similarity solutions for systems arising from an {\it Aedes aegypti} model, Commun. Nonlin. Sci. Numer. Simul., vol. 19, 872--879, (2014).

\bibitem{howard} C. R. Howard, Aedes mosquitoes and Zika virus infection: an A to Z of emergence, Emerging Microbes and Infections, (2016), DOI:10.1038/emi.2016.37.
 
 \bibitem{i2} N. H. Ibragimov, CRC Handbook of Lie group analysis of differential equations, vol. 1, CRC Press, (1994).
 
\bibitem{i1} N. H. Ibragimov, Elementary Lie Group Analysis and Ordinary Differential Equations, John Wiley and Sons, Chirchester (1999).




\bibitem{maidana2008} N. A. Maidana and H. M. Yang, Describing the geographic spread of dengue disease by traveling waves, Math. Biosci., vol. 215, 64--77, (2008).

\bibitem{murray} J. D. Murray, Mathematical biology, Springer, 3th edition, (2002).

\bibitem{olver} P. J. Olver, Applications of Lie groups to differential equations, Springer-Verlag, 1st, edition, (1986).

\bibitem{maidana2005} L. T. Takahashi, N. A. Maidana, W. C. Ferreira Jr., P. Pulino and H. M. Yang, Mathematical models for the {\it Aedes aegypti} dispersal dynamics: traveling waves by wind and wind, Bull. Math. Biol., vol. 67, 509--528, (2005).


\end{thebibliography}
\end{document}